\documentclass[11pt]{article}

\usepackage{float}
\usepackage[letterpaper, left=3.2cm,top=2.5cm,right=3.2cm,bottom=3.2cm]{geometry}

\usepackage[utf8]{inputenc}

\usepackage{authblk}
\usepackage[onehalfspacing]{setspace}

\usepackage{subfig,epsfig,tikz,float}		 
\usepackage{booktabs,multicol,multirow,tabularx}

\usepackage{epstopdf}
\usepackage[english]{babel}
\usepackage{amsmath,amsthm,amssymb}
\usepackage{dsfont}
\usepackage{mathtools}
\usepackage{enumitem}

\newcommand*{\R}{\mathds{R}}

\newcommand{\Ro}{\mathbf{R}}

\newcommand{\supp}{\operatorname{supp}}

\newtheorem{theorem}{Theorem}
\newtheorem{remark}[theorem]{Remark}

\newtheorem{problem}[theorem]{Problem}

\DeclarePairedDelimiter{\abs}{\lvert}{\rvert}
\DeclarePairedDelimiter{\norm}{\lVert}{\rVert}

\DeclareGraphicsExtensions{.eps,.pdf,.png,.jpg}

\newcommand{\la}{\lambda}

\newcommand{\rec}{\mathbf R}
\newcommand{\kernel}{\mathbf k}
\newcommand{\signal}{\mathbf x}
\newcommand{\data}{\mathbf y}

\newcommand{\noise}{{\boldsymbol \delta}}

\newcommand{\fourier}{\mathcal{F}}
\newcommand{\ew}{\mathbb{E}}

\newcommand{\loss}{\mathcal{L}}
\newcommand{\lossA}{\mathcal{L}_A}
\newcommand{\lossB}{\mathcal{L}_B}
\newcommand{\lossC}{\mathcal{L}_C}

\newcommand{\err}{\operatorname{L2err}}
\newcommand{\mcn}{\operatorname{MCN}}

\newcommand{\mnorm}[1]{\bigl \lVert #1 \bigr \rVert}

\DeclareMathOperator*{\argmin}{arg\,min}

\author{Markus Haltmeier}

\affil{Department of Mathematics, University of Innsbruck\authorcr
Technikerstrasse 13, 6020 Innsbruck, Austria \authorcr
E-mail: \texttt{markus.haltmeier@uibk.ac.at}}

\author{Gyeongha Hwang}

\affil{Department of Mathematics, Yeungnam University\authorcr
280 Daehak-Ro, Gyeongsan, Gyeongbuk 38541, South Korea\authorcr
E-mail: \texttt{ghhwang@yu.ac.kr}}

\allowdisplaybreaks

\title{ECALL: Expectation-calibrated learning for unsupervised blind deconvolution}

\date{January 29, 2024}

\numberwithin{equation}{section}
\numberwithin{figure}{section}
\numberwithin{theorem}{section}

\allowdisplaybreaks

\begin{document}

\maketitle

\begin{abstract}
Blind deconvolution aims to recover an original image from a blurred version in the case where the blurring kernel is unknown. It has wide applications in diverse fields such as astronomy, microscopy, and medical imaging. Blind deconvolution is a challenging ill-posed problem that suffers from significant non-uniqueness. Solution methods therefore require the integration of appropriate prior information. Early approaches rely on hand-crafted priors for the original image and the kernel. Recently, deep learning methods have shown excellent performance in addressing this challenge. However, most existing learning methods for blind deconvolution require a paired dataset of original and blurred images, which is often difficult to obtain. In this paper, we present a novel unsupervised learning approach named ECALL (Expectation-CALibrated Learning) that uses separate unpaired collections of original and blurred images. Key features of the proposed loss function are cycle consistency involving the kernel and associated reconstruction operator, and terms that use expectation values of data distributions to obtain information about the kernel. Numerical results are used to support ECALL. 


\end{abstract}

\section{Introduction}\label{sec:1}

Blind deconvolution addresses the problem of recovering the original unknown potentially vector valued image $\signal \colon \R^d \to \R^c$ from its blurry and noisy observation
\begin{align} \label{eq:ip}
	\data = \kernel^\star \ast \signal + \noise \,,
\end{align}
where $\kernel^\star \colon \R^d \to \R$ is the unknown convolution kernel and $\noise$ is the unknown noise. Since the kernel needs to be at least partially estimated along with the recovery of the original image, the problem \eqref{eq:ip} is severely underdetermined and ill-posed. Thus, strong additional assumptions on the original image and the kernel are required for a proper solution. While classical variational methods have been developed for this purpose \cite{chan1998total,lv2018convex,kundur1996blind,marquina2009nonlinear}, in this work we introduce a novel unsupervised learning approach to tackle this challenging problem.

Among others, blind deconvolution is relevant in medical imaging, astronomy, microscopy, signal processing, radar imaging, remote sensing, and computer vision (see \cite{kundur1996blind} and references therein). In astronomy, for example, blind deconvolution is used to improve the resolution of images obtained by telescopes. In medical imaging, it is used to remove blur caused by motion or the imaging process. In microscopy, the image of a specimen is often blurred due to imperfections in the optics, aberrations, and scattering of light. Non-blur deconvolution is simpler and better understood, but requires precise knowledge of the blurring kernel. In the above applications, however, the kernel is often at least partially unknown.

\subsection{Prior work}

Blind deconvolution has been studied extensively for several decades. Due to its highly ill-posed nature, blind deconvolution is usually approached with strong additional assumptions on $\signal$ and $\kernel^\star$ (see \cite{cho2009fast, gong2016blind, krishnan2009fast, krishnan2011blind, levin2009understanding, levin2011efficient, shan2008high, pan2016blind, wipf2014revisiting, xu2010two, xu2013unnatural,lv2018convex}). Typically, $\signal$ is assumed to have a sparse representation in some basis or dictionary, and $\kernel^\star$ is assumed to be a smooth function or to have a compact support. In general, variational regularization \cite{scherzer2009variational} provides a general framework for integrating prior information in blind deconvolution. However, they typically use hand-crafted priors that are not well represented in natural images. They also require time-consuming iterative minimization and may suffer from local minima.

Deep learning based approaches have recently shown excellent results for blind deconvolution tasks \cite{chakrabarti2016neural, gong2017motion, kupyn2018deblurgan, nah2017deep, noroozi2017motion, ramakrishnan2017deep, sun2015learning, schuler2015learning, tao2018scale, zhang2018dynamic}. Most of these approaches are supervised and require a paired dataset of original and blurred images. Existing methods fall into two broad categories: Two-step approaches with kernel estimation in a first step, and end-to-end approaches for estimating only the $\signal$. See \cite{zhang2022deep} for a detailed review of such supervised learning approaches in blind deconvolution. Paired data however is in many cases difficult to obtain. 

In contrast, unsupervised methods overcome the challenging problem of collecting corresponding pairs of original and blurred images. For example, in \cite{ren2020neural} a method inspired by the principles of Deep Image Prior (DIP) and Double-DIP \cite{gandelsman2019double,baguer2020computed,ulyanov2018deep} is presented. In \cite{li2022supervised}, a method is proposed to approximate the maximum a posteriori estimate of the blurring kernel using Monte Carlo sampling. Another approach \cite{lu2019unsupervised} is based on Generative Adversarial Networks (GAN) \cite{creswell2018generative}. While GAN-based methods suffer from the instability of GAN training, the DIP-based method suffers from long inference time and sensitivity to hyperparameters.

\subsection{Main contributions}

In this paper, we present a novel simple unsupervised learning approach to blind deconvolution that overcomes the above problems. Specifically, we consider blind deconvolution with an unknown kernel $\kernel^\star$ when separate unpaired collections of original and blurred images are given. It is proposed to estimate an unknown kernel $\kernel^\star$ and the reconstruction operator $\rec^\star \colon \data \to \signal$ via unsupervised learning by comparing expectations of the Fourier transform of given seperate unpaired collections of original and blurred images along with a cycle consistency term. Our method is particularly useful in various applications where the acquisition of a paired dataset is challenging. As our main contribution, we motivate and introduce a novel loss function that estimates the kernel and reconstruction operator. The loss function includes a cycle constancy term, which requires $\rec^\star$ to be a near-inverse of the convolution with $\kernel^\star$, and an expectation calibration term to constrain $\kernel^\star$. We refer to the proposed method as ECALL (Expectation-CALibrated Learning). We present numerical results that demonstrate the success of ECALL.

\subsection{Outline}

The rest of the paper is organized as follows: In Section~\ref{sec:method} we present the problem formulation and our proposed unsupervised learning method ECALL for blind deconvolution. In particular, the design of the proposed loss function is presented in detail. Numerical aspects and numerical results are presented in Section~\ref{sec:numerics}. The paper ends with conclusions and an outlook on future work presented in Section~\ref{sec:conclusion}.

\section{Proposed ECALL method}
 \label{sec:method}

\subsection{Problem formulation}

We consider the blind deconvolution problem \eqref{eq:ip} in the case when $\signal$ and $\noise$ are independent random variables in $L^2(\mathbb{R}^d)$ modelling original images and noise, respectively, $\kernel^\star \in L^1(\mathbb{R}^d)$ is a deterministic kernel and $\data$ are the given data representing noisy blurred images. We assume the noise to have zero mean. We write $\fourier$ and $\fourier^{-1}$ for the Fourier transform on $\R^d$ and its inverse, and write $\ew[\cdot]$ for the expectation. At some places we will use the abbreviation $\widehat \signal \coloneqq \fourier (\signal)$. We assume that the convolution is well defined by $ \fourier(\kernel^\star \ast \signal) = \fourier(\kernel^\star) \cdot \fourier (\signal)$ and given by point-wise multiplication in Fourier space.

Let $p_\signal$, $p_{\data}$ and $p_\noise$ be the distributions of $\signal$, $\data$ and $\noise$, respectively. Our goal is to find the kernel $\kernel^\star$ together with a reconstruction operator $\rec^\star \colon \data \to \signal$ mapping  noisy blurred images to original ones based on \eqref{eq:ip}. More precisely, we consider the following problem.

\begin{problem}[Unsupervised blind deconvolution] \label{prob:BID}
For given distributions $p_\signal$, $p_{\data}$ and $p_\noise$ subject to model \eqref{eq:ip} determine the kernel $\kernel^\star$ and the reconstruction operator 
\begin{equation} \label{eq:rec}
	 \rec^\star \coloneqq \argmin_\rec \ew \bigl[ \norm{ \rec(\kernel^\star \ast \signal + \noise) - \signal}^2 \bigr] \,.
\end{equation}
\end{problem}

Since, in practice, only a finite date set can be collected, we will actually address the following empirical version.

\begin{problem}[Empirical unsupervised blind deconvolution]\label{prob:BID-E}
For given unpaired collections of original images $(\signal_i)_{i =1, \dots, N} \sim p_\signal$, data $(\data_i)_{i =1, \dots, N} \sim p_{\data}$ and noises $(\noise_i)_{i =1, \dots, N} \sim p_\noise$, estimate the kernel $\kernel^\star$ and the reconstruction operator $\rec^\star$ defined by \eqref{eq:rec}.
\end{problem}

Note that in Problem \ref{prob:BID-E} we assume no dependence among $(\signal_i)_{i=1}^N$, $(\data_i)_{i=1}^N$ and $(\noise_i)_{i=1}^N$.

\begin{remark}
In the supervised learning setting \cite{vapnik1999nature} the ideal reconstruction operator $\rec^\star$ is given as the minimizer of the expected loss $ \ew [ \norm{ \rec(\data) - \signal}^2 ] $. It implementation, however, requires paired data, which we don't have. If $\kernel^\star$ would be known, we could create paired data using samples of $\signal$ and the model $\kernel^\star \ast \signal + \noise$. Our strategy will therefore include kernel estimation, which will then allow the creation of virtual supervised training data based on the kernel estimate. 
\end{remark}

To solve Problem~\ref{prob:BID}, we set up an expected loss functional that determines $\kernel^\ast$ based on $p_\signal$ and $p_{\data}$ (expectation calibration) and $\rec^\ast$ based on virtual supervised data using the estimated kernel. To stabilize the whole process, we integrate several additional terms. In particular, the joint estimation of $\kernel^\star$ and $\rec^\star$ improves the estimation of each. Problem \ref{prob:BID-E} then uses the empirical version of the expected loss together with specific network architectures for approximating $(k^\star, \rec^\star)$.

\subsection{Theory}
\label{sec:noise-free}

Let us start with the theoretical Problem~\ref{prob:BID}. Our ECALL method is motivated by the following simple result, which constitutes a constructive uniqueness theorem for the problem of determining $(k^\star, \rec^\star)$.

\begin{theorem}[Uniqueness results] \label{thm:main}
Suppose $\ew[ \widehat \signal] (\xi) \neq 0$ for almost every $ \xi \in \supp(\widehat {\kernel^\star})$. Then $(\kernel^\star, \rec^\star)$ is the unique solution of 
\begin{align} \label{eq:determine1}
 \kernel^\star &= \argmin_\kernel 
 \norm{\ew [ \widehat \data ] - \ew[ \mathcal{F} (\kernel \ast \signal) ] }_1 
 \\ \label{eq:determine2}
 \rec^\star &= \argmin_\rec 
 \ew \norm { \signal - \rec( \kernel^\star \ast \signal + \delta) }^2_2 \,.
\end{align}
\end{theorem}

\begin{proof}
By the convolution theorem and equation \eqref{eq:ip} we have $\widehat\data = \widehat{\kernel^\star} \cdot \widehat \signal + \widehat \noise$. By applying expectation values, using that $\kernel^\star$ is a deterministic quantity, and that the noise has zero mean and is independent of the clean image, we obtain \eqref{eq:determine1}. Having determined the kernel $\kernel^\star$, identity \eqref{eq:determine2} is then simply the definition of $\rec^\star$. 
\end{proof}

Based on Theorem \ref{thm:main}, we set up a loss function that includes terms for estimating the kernel $\kernel^\star$ and terms for estimating the reconstruction operator $\rec^\star$. To increase accuracy and stability, we introduce regularization.
More specifically, we consider 
\begin{equation} \label{eq:loss}
\loss (\kernel, \rec) =
\lossA (\kernel)
+
\lossB (\kernel, \rec)
+
\lossC (\kernel, \rec)
\end{equation}
where the three terms $\lossA$, $\lossB$, $\lossC$ are described next.

\begin{enumerate}[label=(\Alph*), topsep = 0em, leftmargin = 3em]
\item 
\textsc{Expectation calibration:}
 Based~on \eqref{eq:determine1} we introduce a term for determining the kernel $\kernel^\star$. Under the assumption that $\ew[\widehat \signal] (\xi) \neq 0$, minimizing the term $\norm{\ew [ \widehat \data ] - \ew[ \mathcal{F} (\kernel \ast \signal + \delta) ] }_1$ is theoretically sufficient for that purpose. However, when $\ew[\widehat \signal] $ is close to zero, then the estimation becomes unstable. We thus add a second term involving expected values of Fourier magnitudes. Thus we take 
 \begin{equation}
\label{eq:LA}
\lossA(\kernel)
=
\la_{A,1} \mnorm{ \ew [ \widehat\data ] - \ew [ \mathcal{F}( \kernel \ast \signal + \noise) ] }_1
+ \la_{A,2} \mnorm{ \ew [ \abs{\widehat\data }] - \ew [ \abs{\mathcal{F}(\kernel \ast \signal+ \noise) }] }_1 \,.
\end{equation}
 Loss \eqref{eq:LA} is the main term allowing to estimate the unknown convolution kernel. While theoretically the first term would be sufficient the second term turns out to significantly increase stability.

\item 
\textsc{Cycle consistency:}
Based on \eqref{eq:determine2} we construct a term to obtain the reconstruction operator $\rec^\star$. 
If the exact kernel $\kernel = \kernel^\star$ would be known, the term $\ew [ \norm{\signal - \Ro (\kernel \ast \signal + \noise)}^2] $ would be sufficient for that purpose. However, as the kernel is estimated simultaneously with the reconstruction operator we found that adding $\ew [ \norm{\data - \kernel \ast (\Ro \data)}^2]$ significantly improves results and, in particular, also stabilizes the kernel estimation. Thus we consider the term 
\begin{equation} \label{eq:LB}
\lossB(\kernel, \rec) =
\lambda_{B,1} \ew [ \norm{\data - \kernel \ast (\Ro \data)}^2] +
\lambda_{B,2} \ew [ \norm{\signal - \Ro(\kernel \ast \signal + \noise ) }^2] \,.
\end{equation}
It implements the reconstruction property \eqref{eq:determine2} together with data consistency. Making $\lossB$ small requires $ \kernel \ast (\Ro \data) \simeq \data$ and $ \Ro (\kernel \ast \signal) \simeq \signal$ and thus resembles the cycle loss commonly considered for unpaired image-to-image translation \cite{zhu2017unpaired}.

\item 
\textsc{Regularization:}
To avoid overfitting, we consider simple $L^2$-regularization $\norm{\kernel}_2^2$. To stabilize estimating the reconstruction operator, we add the difference between the expectations of $\signal$ and $\Ro\data$ resulting in 
\begin{equation} \label{eq:LC}
\lossC(\kernel, \rec) = \la_{C,1} \norm{\kernel}_2^2 + \la_{C,2} \norm{\ew[\Ro \data] - \ew[\signal] }_2^2 \,.
\end{equation}
Together with the expectation calibration \eqref{eq:LA} and cycle consistency \eqref{eq:LB} the regularization term \eqref{eq:LC} forms the loss function \eqref{eq:loss}.
\end{enumerate}

\subsection{Empirical estimation}

Next we turn over to the practically more important Problem~\ref{prob:BID-E} of blind deconvolution using empirical data $(\signal_i)$, $(\data_i^\noise)$ and $(\noise_i)$. For that purpose we replace all expectation values in \eqref{eq:loss} by empirical counterparts. The resulting functional is minimized over a parameterized class of convolution kernels $(\kernel_w)_{w \in W}$ for the kernel and a convolutional neural network $(\rec_\theta)_{\theta \in \Theta}$ for the reconstruction operator. Thus we consider 
\begin{equation} \label{eq:E-loss}
\loss^N (w, \theta) =
\lossA^N (w)
+
\lossB^N (w, \theta)
+
\lossC^N (w, \theta)
\end{equation}
where
\begin{align} \nonumber
\lossA^N (w)&=
\frac{\la_{A,1}}{N}
\Bigl\lVert \sum_{i=1}^N \widehat{\data_i^\delta} - \sum_{i=1}^N \mathcal{F}(\kernel_w \ast \signal_i + \delta_{\sigma(i)}) \Bigr\rVert_1
 \\ \label{eq:E-LA} &\qquad\qquad +
\frac{\la_{A,2}}{N}
\Bigl\lVert \sum_{i=1}^N \abs{ \widehat{\data_i^\delta}} - \sum_{i=1}^N \abs{ \mathcal{F}(\kernel_w \ast \signal_i + \delta_{\sigma(i)}) } \Bigr\rVert_1
 \\ \nonumber
\lossB^N (w, \theta) &=
 \frac{\lambda_{B,1}}{N} \sum_{i=1}^N 
 \Bigl\lVert \data_i^\delta - \kernel_w \ast ( \Ro_\theta \data_i^\delta) \Bigr\rVert_2^2
 \\ \label{eq:E-LB} &\qquad 
 + \frac{\lambda_{B,2}}{N}\sum_{i=1}^N 
 \Bigl\lVert \signal_i - \Ro_\theta (\kernel_w \ast \signal_i + \delta_{\sigma(i)})\Bigr\rVert_2^2
 \\ \label{eq:E-LC} 
\lossC^N (w, \theta) &= \la_{C,1} \norm{w}_2^2 + 
\frac{\la_{C,2}}{N} \Bigl\lVert \sum_{i=1}^N \signal_i - \sum_{i=1}^N \Ro_\theta(\data_i^\delta)\Bigr\rVert_2^2 \,.
\end{align}
Here $(\delta_{\sigma(i)})$ denotes random sampling from the given collection of noises.

Details on the implementation and the network architecture are given below.

\section{Numerical simulations}
\label{sec:numerics}

In this section, we present details of implementation and experimental results. We will work with color images with size of $256 \times 256$ having three color channels. The convolution kernel $\kernel^\star$ is approximated by a parameterised family of convolutions $(\kernel_w)_{w \in W}$ realised as CNN with one convolution layer with kernel size $31 \times 31$. The entries of the kernel are taken as parameters of the network. The reconstruction operator $\rec^\star$ is approximated by the U-Net architecture $(\Ro_\theta)_{\theta \in \Theta}$ which has proven to be highly effective for a variety of image processing tasks \cite{ronneberger2015u}. The U-Net structure used in our study is shown in Figure \ref{fig:unet_design}.

 \begin{figure}[htb!]
 \centering
 \includegraphics[width=0.9\textwidth]{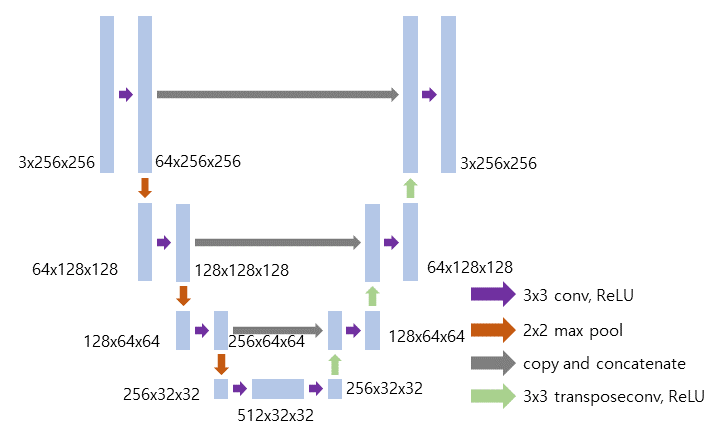}
 \caption{U-net architecture $(\Ro_\theta)_{\theta \in \Theta}$ for representing the reconstruction operator $\rec^\star$ used in all presented numerical results.}
 \label{fig:unet_design}
\end{figure}

\subsection{Datasets and data generation}

We perform experiments using $3*10^3$ images of the FFHQ dataset \cite{karras2019style} where $10^3$ samples are taken as original images, other $10^3$ images to generate data samples and another $10^3$ images as taken as test set.

Blurry images are created by applying three Gaussian filters, referred to as Broad, Medium, and Narrow. The window size of each Gaussian filter is truncated to have size $31$. The Gaussian kernels are centered at zero and have different standard deviation, which determines the degree of blurring applied to the image. The Broad filter has the smallest standard deviation of $0.5$, which will result in a mild blurring effect. The Medium filter has a standard deviation of $1$, which will result in a moderate blurring effect. Finally, the Narrow filter has the largest standard deviation of $2$, which will result in a strong blurring effect. For the noisy data, we add additive Gaussian white noise with standard deviation equal to $1\%$ of the maximal value of the original data. The Fourier transform in the loss function is realized by the FFT applied with periodic padding \cite{liu2008reducing}.

\subsection{Training}

For minimizing \eqref{eq:E-loss} we use the AdamW optimizer \cite{loshchilov2017decoupled} which is a stochastic optimization method that modifies the conventional weight decay implementation in the Adam \cite{kingma2014adam} by decoupling weight decay from the gradient update process. The batch size for the neural network $\kernel_w$ is taken as $10^3$. The large batch helps to evaluate and update the expectation effectively. We train the neural network $\Ro_\theta$ using a batch size of 2, which helps to reduce the gradient variance and improve model generalization.

The specific training strategy based on \eqref{eq:E-loss} consists of three phases: 
\begin{enumerate}[label = {[}Phase \arabic*{]}, leftmargin=5em, topsep=0em]
 \item \label{phase1} We first train the kernel $\kernel_w$ with the objective function \eqref{eq:E-loss} using coefficients $\lambda_{A,1}=1$, $\lambda_{C,1}=5$ and $\lambda_{A,2} = \lambda_{B,1} = \lambda_{B,2} = \lambda_{C,2}=0$ with the learning rate of $10^{-3}$ for $\kernel_w$. The iteration number is $10^3$.
 
 \item\label{phase2} Train the neural networks $\kernel_w$ and $\Ro_\theta$ with the objective function \eqref{eq:E-loss} which the coefficients $\lambda_{A,1} = \lambda_{A,2}=10 $, $\lambda_{B,1} = \lambda_{B,2}=1$, $\lambda_{C,1}=5$ and $\lambda_{C,2}=10$ with the learning rate of $10^{-4}$ for $\kernel_w$ and $10^{-3}$ for $\Ro_\theta$. The iteration number is $10^4$.
 
 \item\label{phase3} In the final fine-tuning phase we train the neural network $\Ro_\theta$ with the objective function \eqref{eq:E-loss} with $\lambda_{A,1} = \lambda_{A,2} = \lambda_{B,1} = \lambda_{C,1} = 0$, $ \lambda_{B,2}=1$ and $\lambda_{C,2}=10$ with the learning rate of $10^{-3}$ for $\Ro_\theta$. The iteration number is $10^4$.
\end{enumerate}

To prevent getting stuck in a local minimum, we modify $\lossA^N$ as 
\begin{multline} 
\lossA^N (w)=
\frac{\la_{A,1}}{N}
\Bigl\lVert \sum_{i=1}^N \chi \odot \widehat{\data_i^\delta} - \sum_{i=1}^N \chi \odot \mathcal{F}(\kernel_w \ast \signal_i + \delta_{\sigma(i)}) \Bigr\rVert_1
 \\ \label{eq:E-LA-M} 
 +
 \frac{\la_{A,2}}{N}
\Bigl\lVert \sum_{i=1}^N \abs{ \chi \odot \widehat{\data_i^\delta}} - \sum_{i=1}^N \abs{ \chi \odot \mathcal{F}(\kernel_w \ast \signal_i + \delta_{\sigma(i)}) } \Bigr\rVert_1
\end{multline}
where $\chi$ is a random mask that sets $20\%$ of the pixels to zero and keeps the other pixels unchanged, and $\odot$ denotes the Hadamard product.

\subsection{Results}

The experiments for proposed unsupervised blind deconvolution are made with noiseless data as well as noisy data. Kernel estimation is the most challenging part. Due to the empirical data this task in noiseless and noisy case suffers from inexact data; see \eqref{eq:E-LA}. For comparison purpose we also present results with supervised learning where we minimise the supervised loss 
\begin{multline*}
\mathcal{L}_{N,\mathrm{super}}(w, \theta) \coloneqq \frac{1}{N}\sum_{i=1}^N\left\|\kernel_w(\signal_i) - (\kernel^\star \ast \signal_i + \noise_i) \right\|_2^2 \\ 
+ \frac{1}{N}\sum_{i=1}^N\left\| \signal_i - \Ro_\theta(\kernel^\star \ast \signal_i + \noise_i)\right\|_2^2 + 5 \| \kernel_w \|_2^2 \,,
\end{multline*} 
using the same architectures as in the unsupervised case, learning rates of $10^{-4}$ and $10^{-3}$ for $\kernel_w$ and $\Ro_\theta$ respectively, and iteration number $2*10^{4}$. 

The estimated kernels in the supervised and unsupervised (ECALL) case are shown in Table~\ref{tbl:kernel}. Both methods are able to well recover the kernel especially the broad and medium kernel. Table \ref{tbl:test} shows typical example of corresponding blind deconvolution of the images. 
\begin{table}[H]
\centering
\begin{tabular}{|m{2cm}|m{1.3cm}m{1.3cm}m{1.3cm}|m{1.3cm}m{1.3cm}m{1.3cm}|}
\toprule
& &\centering Noiseless & & &\centering Noisy & \tabularnewline
\midrule
 &\centering Broad &\centering Medium &\centering Narrow &\centering Broad &\centering Medium &\centering Narrow \tabularnewline
\midrule
Ground Truth & \includegraphics[width=1.3cm]{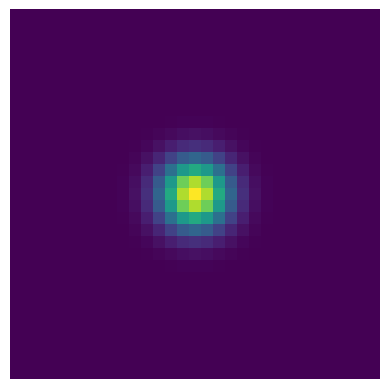} &
\includegraphics[width=1.3cm]{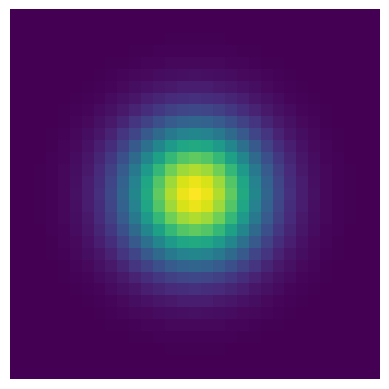} &
\includegraphics[width=1.3cm]{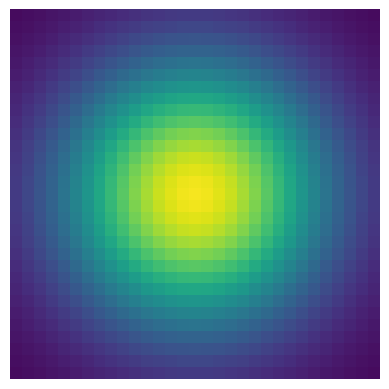} & \includegraphics[width=1.3cm]{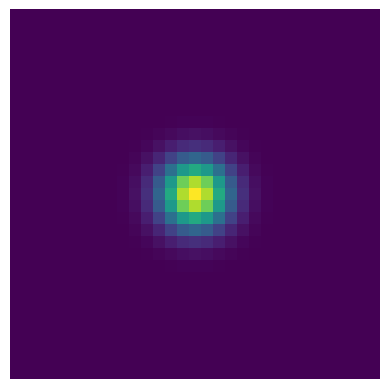} &
\includegraphics[width=1.3cm]{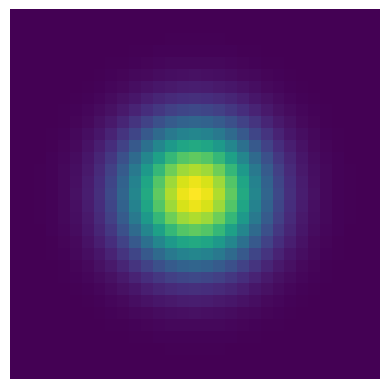} &
\includegraphics[width=1.3cm]{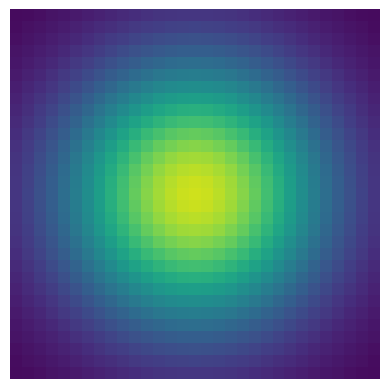} \\
Supervised& \includegraphics[width=1.3cm]{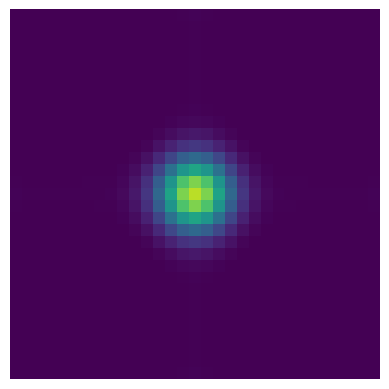} &
\includegraphics[width=1.3cm]{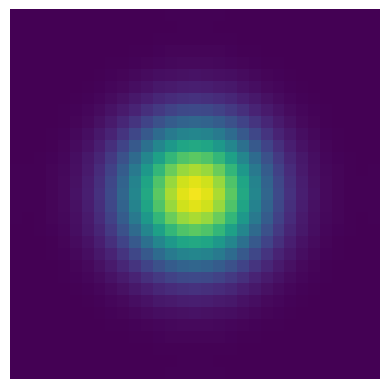} &
\includegraphics[width=1.3cm]{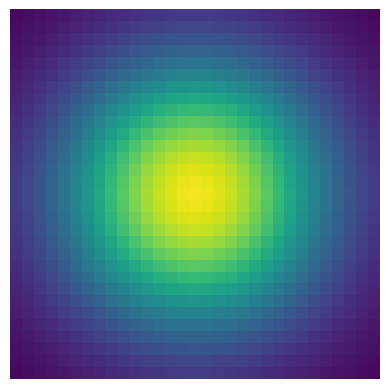} & \includegraphics[width=1.3cm]{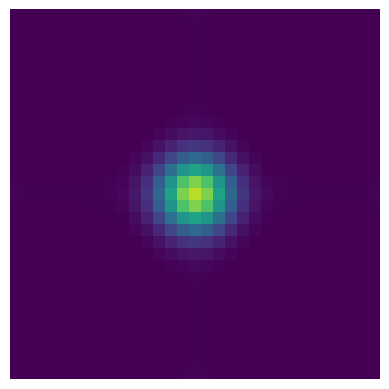} &
\includegraphics[width=1.3cm]{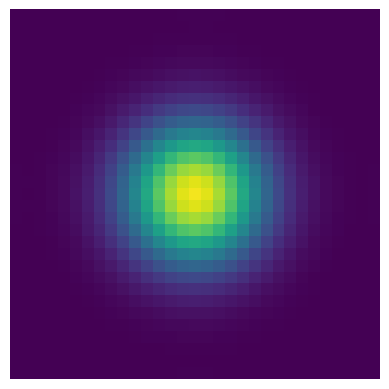} &
\includegraphics[width=1.3cm]{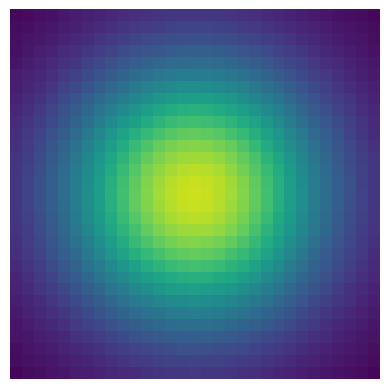}\\
ECALL & \includegraphics[width=1.3cm]{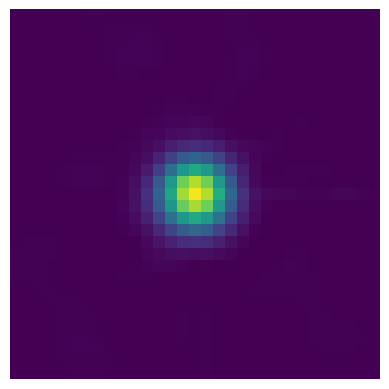} &
\includegraphics[width=1.3cm]{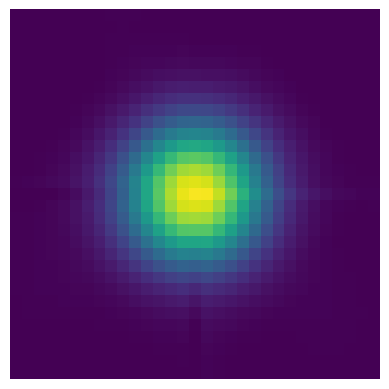} &
\includegraphics[width=1.3cm]{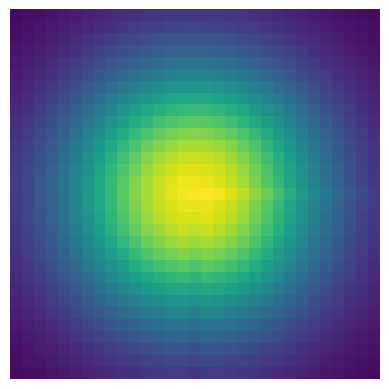} & \includegraphics[width=1.3cm]{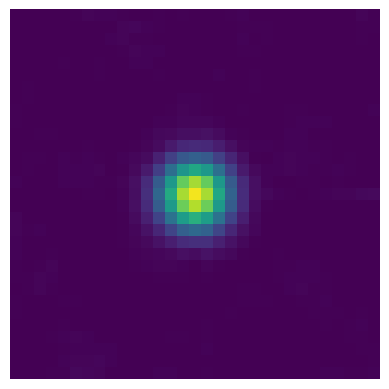} &
\includegraphics[width=1.3cm]{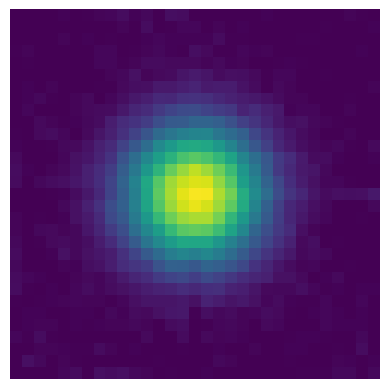} &
\includegraphics[width=1.3cm]{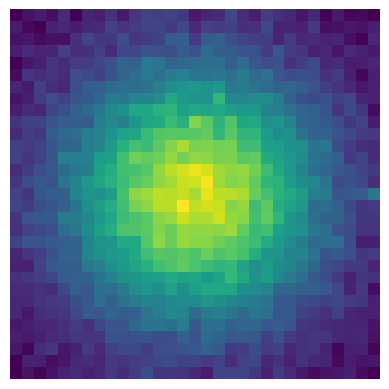}\\
& \includegraphics[width=1.3cm]{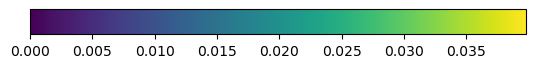} &
\includegraphics[width=1.3cm]{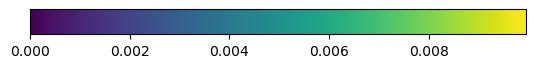} &
\includegraphics[width=1.3cm]{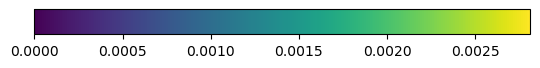}&
\includegraphics[width=1.3cm]{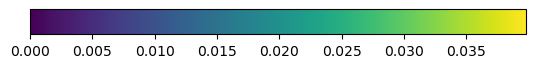} &
\includegraphics[width=1.3cm]{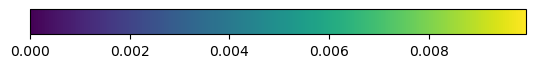} &
\includegraphics[width=1.3cm]{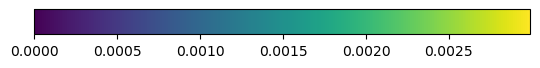}\\
\bottomrule
\end{tabular}
\caption{Ground truth kernels (top) and reconstructed kernels using the supervised (middle) and proposed unsupervised (bottom) learning.}
\label{tbl:kernel}
\end{table}

\begin{table}[H]
\centering
\makebox[\linewidth]{
\begin{tabular}{|m{2cm}|m{1.3cm}m{1.3cm}m{1.3cm}|m{1.3cm}m{1.3cm}m{1.3cm}|}
\toprule
& &\centering Noiseless & & &\centering Noisy & \tabularnewline
\midrule
&\centering Broad &\centering Medium &\centering Narrow &\centering Broad &\centering Medium &\centering Narrow\tabularnewline
\midrule
Original & \includegraphics[width=1.3cm]{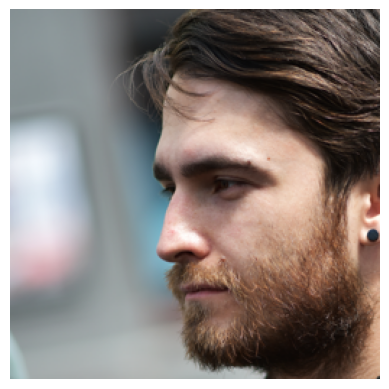} & \includegraphics[width=1.3cm]{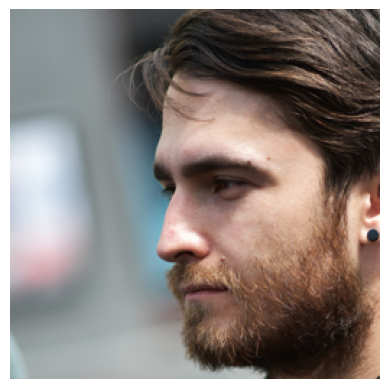} & \includegraphics[width=1.3cm]{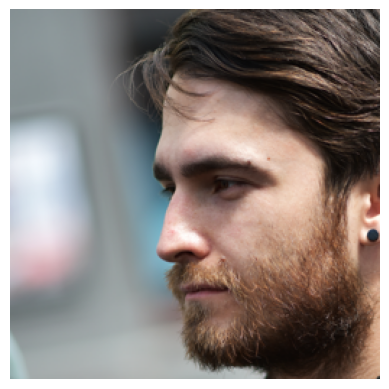} & \includegraphics[width=1.3cm]{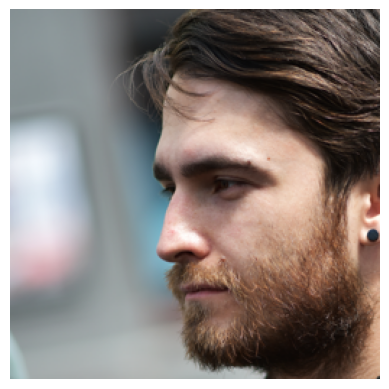} & \includegraphics[width=1.3cm]{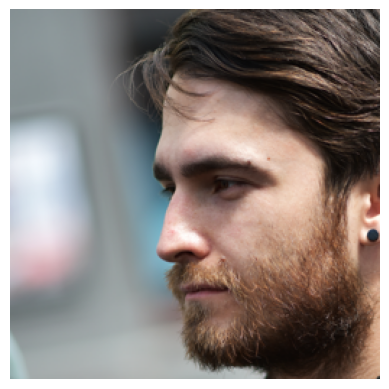} & \includegraphics[width=1.3cm]{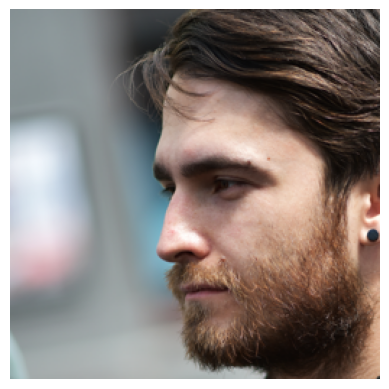}\\
Observed & \includegraphics[width=1.3cm]{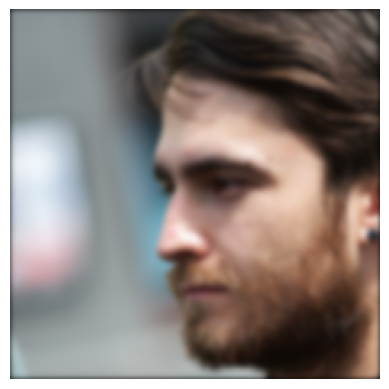} & \includegraphics[width=1.3cm]{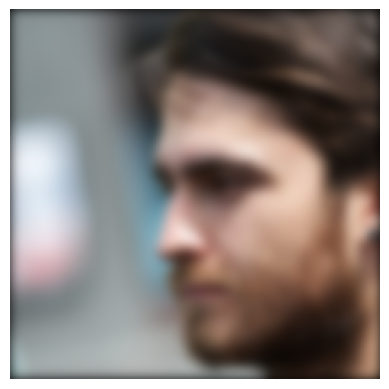} & \includegraphics[width=1.3cm]{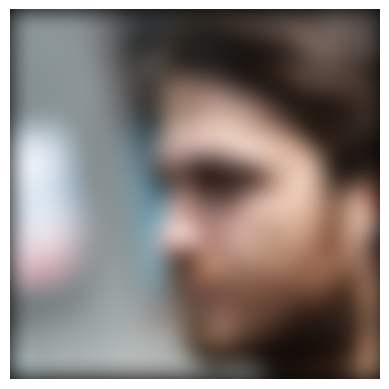} & \includegraphics[width=1.3cm]{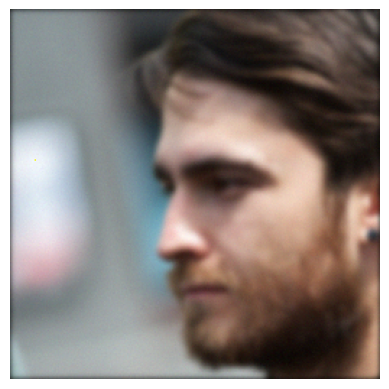} & \includegraphics[width=1.3cm]{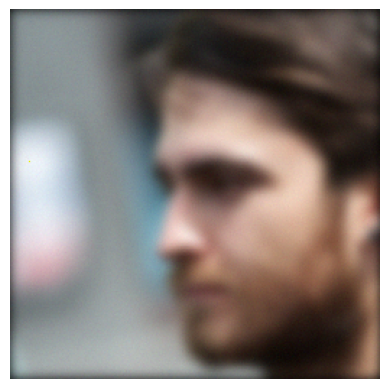} & \includegraphics[width=1.3cm]{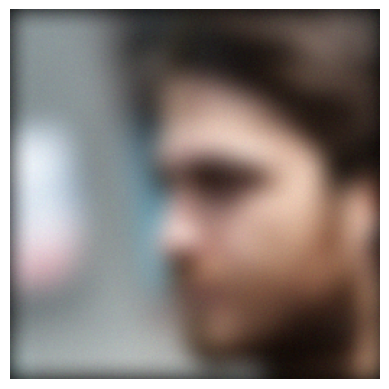}\\
Supervised & \includegraphics[width=1.3cm]{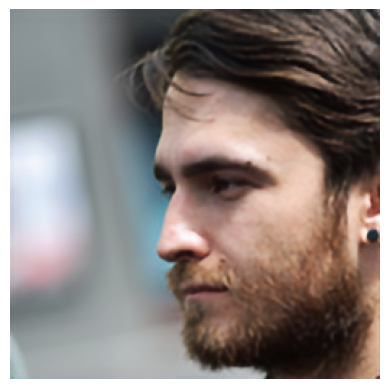} & \includegraphics[width=1.3cm]{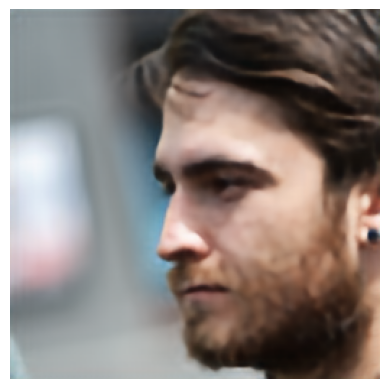} & \includegraphics[width=1.3cm]{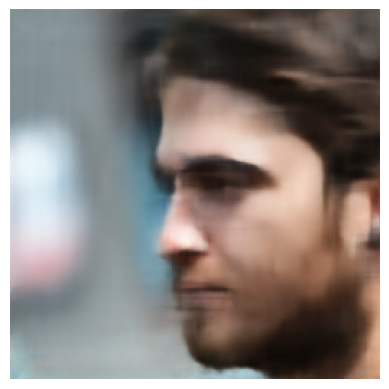} & \includegraphics[width=1.3cm]{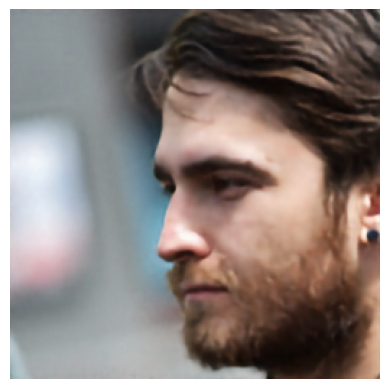} & \includegraphics[width=1.3cm]{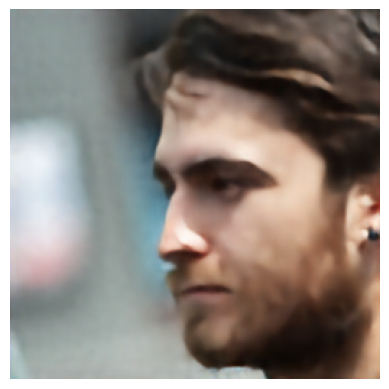} & \includegraphics[width=1.3cm]{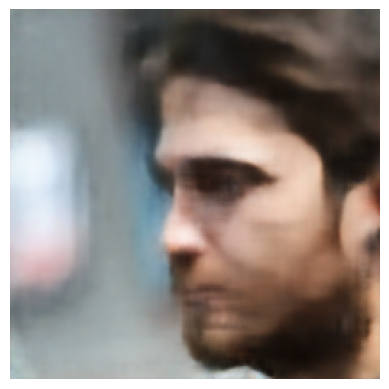}\\
ECALL & \includegraphics[width=1.3cm]{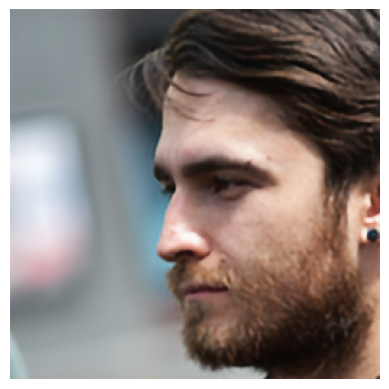} & \includegraphics[width=1.3cm]{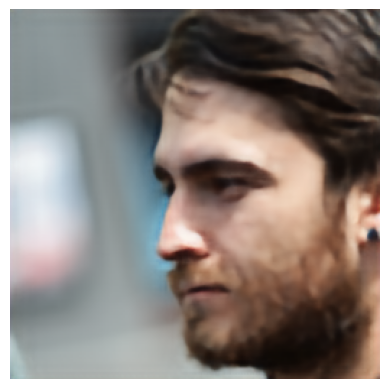} & \includegraphics[width=1.3cm]{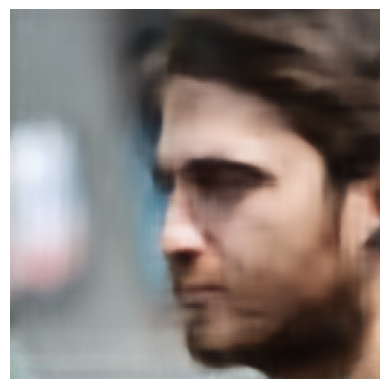} & \includegraphics[width=1.3cm]{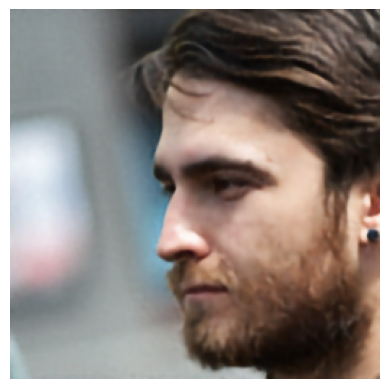} & \includegraphics[width=1.3cm]{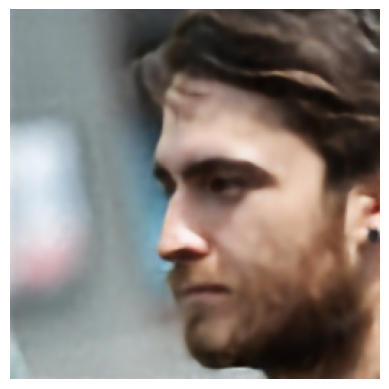} & \includegraphics[width=1.3cm]{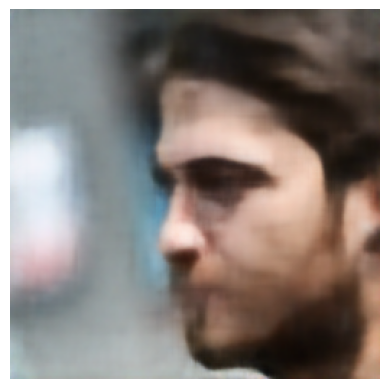} \\
\midrule
\end{tabular}}
\makebox[\linewidth]{
\begin{tabular}
{|m{2cm}|m{1.3cm}m{1.3cm}m{1.3cm}|m{1.3cm}m{1.3cm}m{1.3cm}|}
Original & \includegraphics[width=1.3cm]{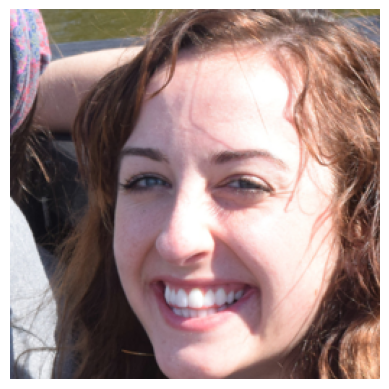} & \includegraphics[width=1.3cm]{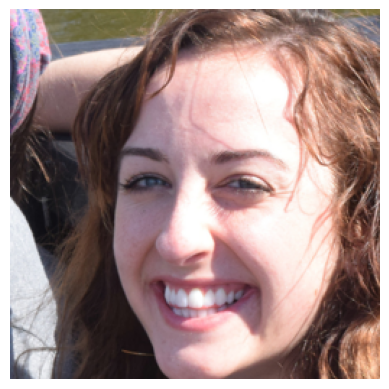} & \includegraphics[width=1.3cm]{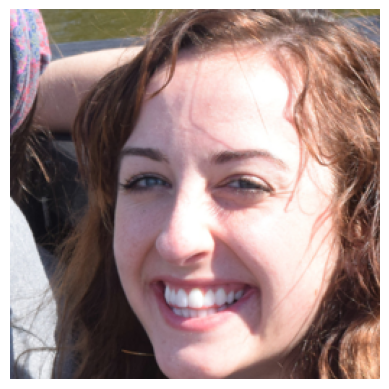} & \includegraphics[width=1.3cm]{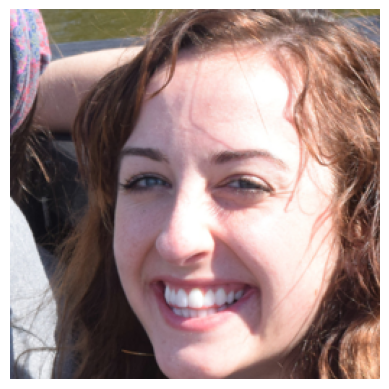} & \includegraphics[width=1.3cm]{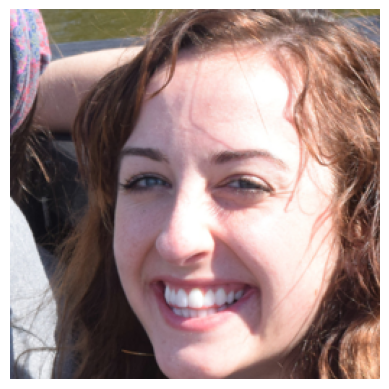} & \includegraphics[width=1.3cm]{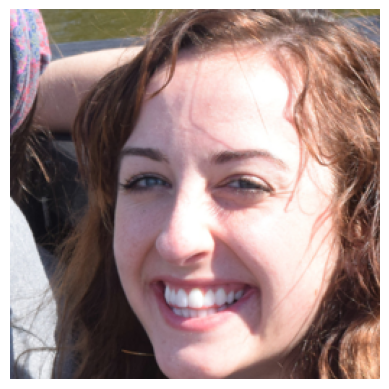}\\
Observed & \includegraphics[width=1.3cm]{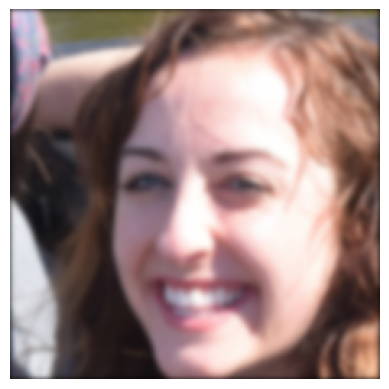} & \includegraphics[width=1.3cm]{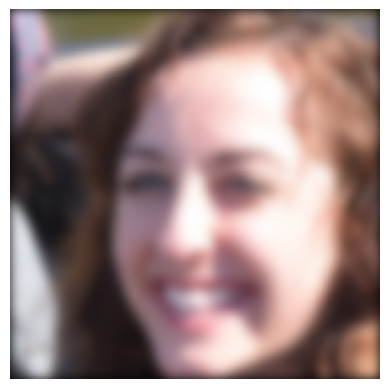} & \includegraphics[width=1.3cm]{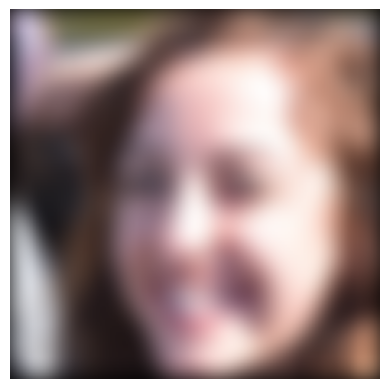} & \includegraphics[width=1.3cm]{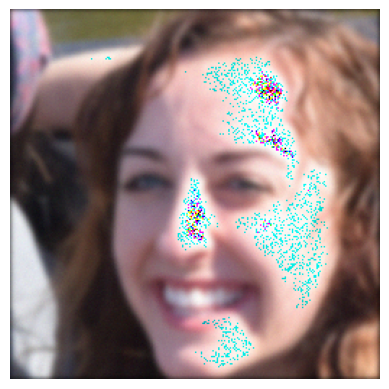} & \includegraphics[width=1.3cm]{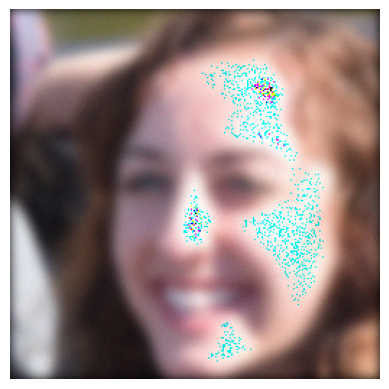} & \includegraphics[width=1.3cm]{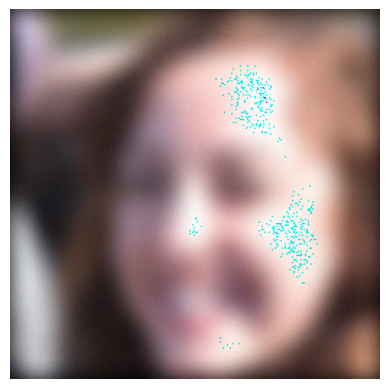}\\
Supervised & \includegraphics[width=1.3cm]{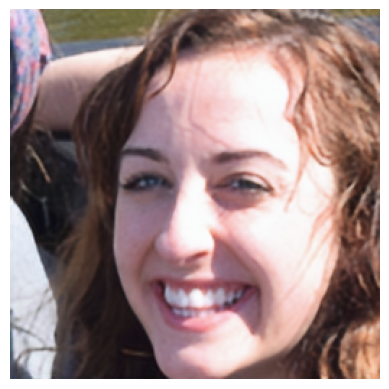} & \includegraphics[width=1.3cm]{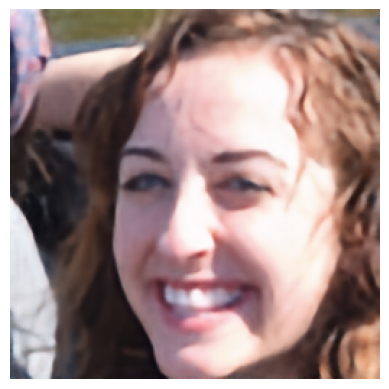} & \includegraphics[width=1.3cm]{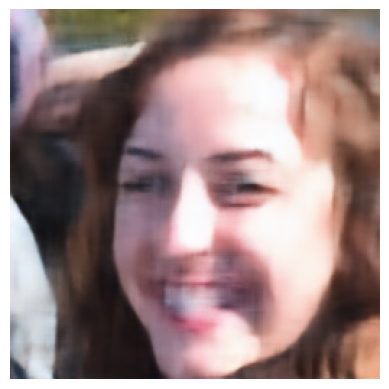} & \includegraphics[width=1.3cm]{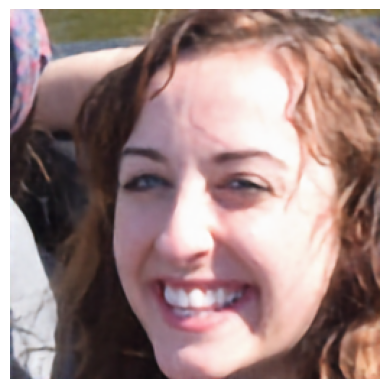} & \includegraphics[width=1.3cm]{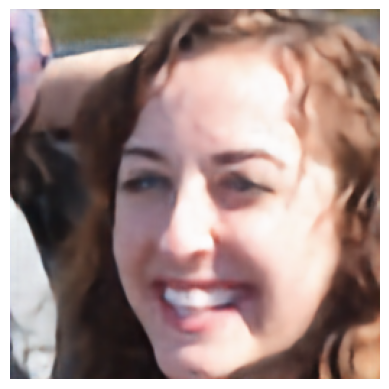} & \includegraphics[width=1.3cm]{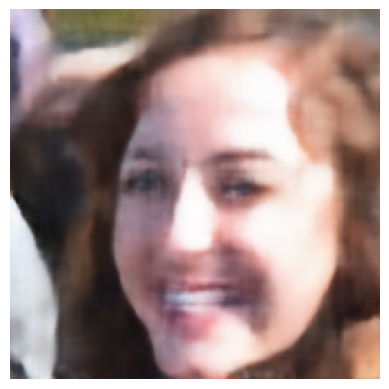}\\
ECALL & \includegraphics[width=1.3cm]{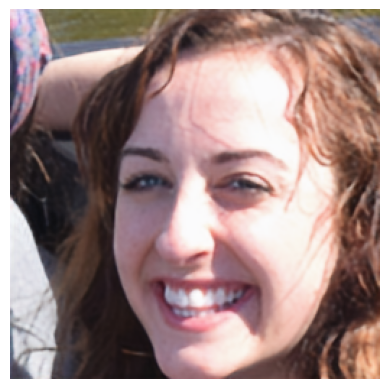} & \includegraphics[width=1.3cm]{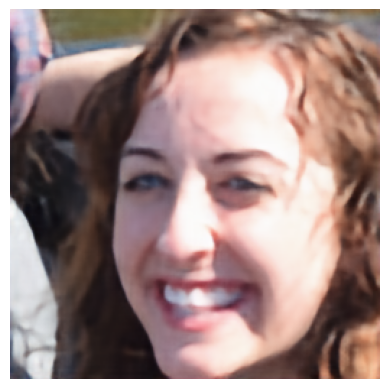} & \includegraphics[width=1.3cm]{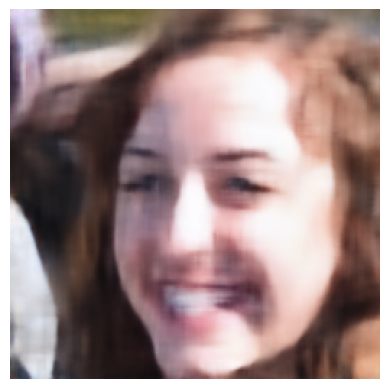} & \includegraphics[width=1.3cm]{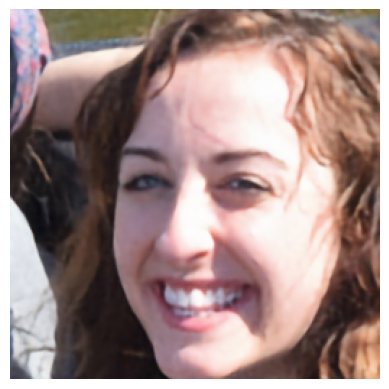} & \includegraphics[width=1.3cm]{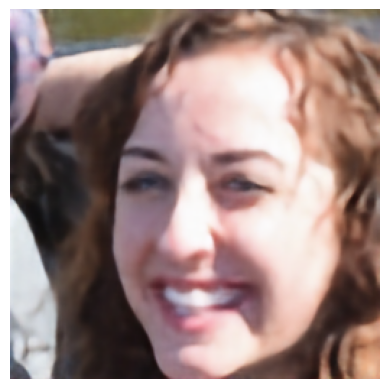} & \includegraphics[width=1.3cm]{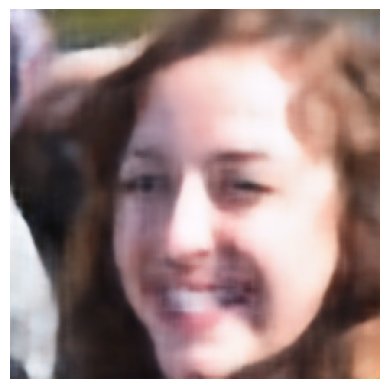} \\
\midrule
\end{tabular}}
\makebox[\linewidth]{
\begin{tabular}
{|m{2cm}|m{1.3cm}m{1.3cm}m{1.3cm}|m{1.3cm}m{1.3cm}m{1.3cm}|}
Original & \includegraphics[width=1.3cm]{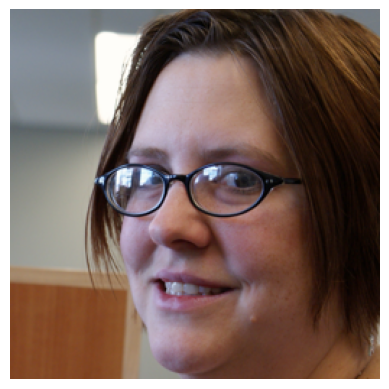} & \includegraphics[width=1.3cm]{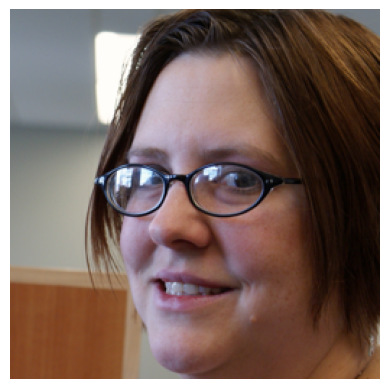} & \includegraphics[width=1.3cm]{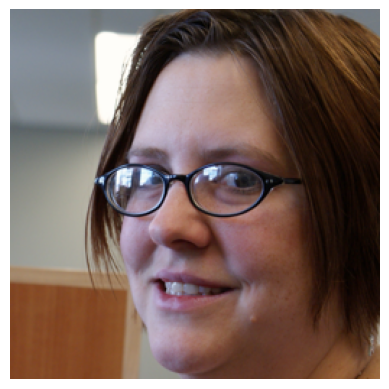} & \includegraphics[width=1.3cm]{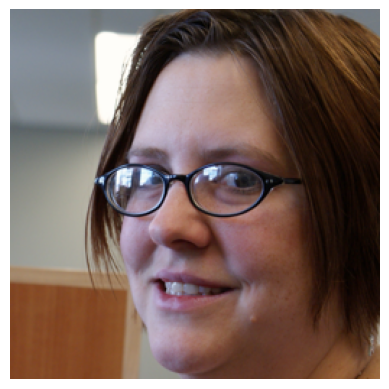} & \includegraphics[width=1.3cm]{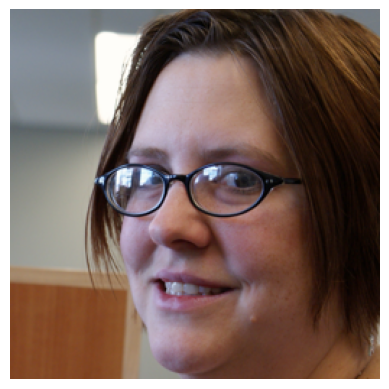} & \includegraphics[width=1.3cm]{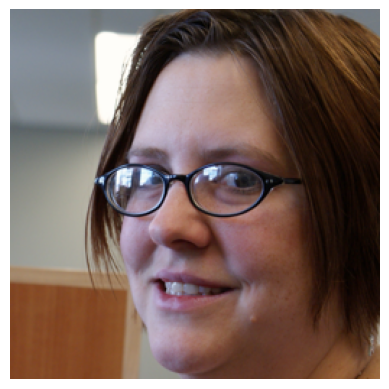}\\
Observed & \includegraphics[width=1.3cm]{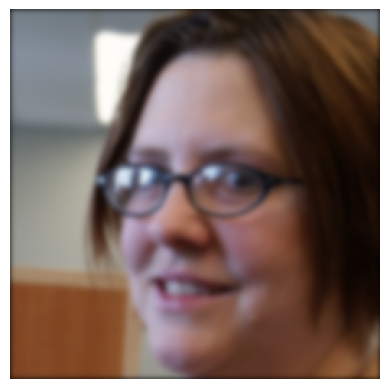} & \includegraphics[width=1.3cm]{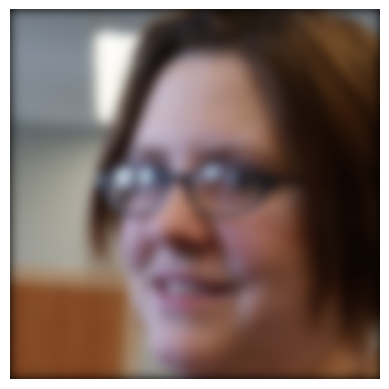} & \includegraphics[width=1.3cm]{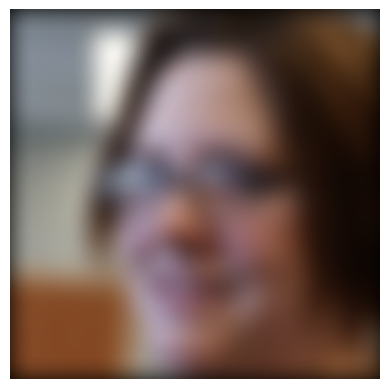} & \includegraphics[width=1.3cm]{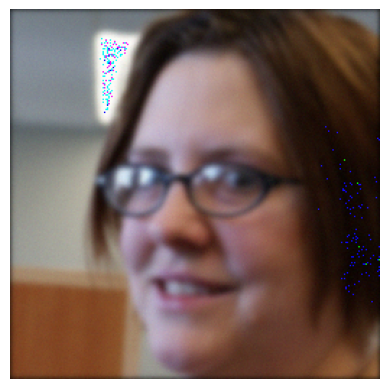} & \includegraphics[width=1.3cm]{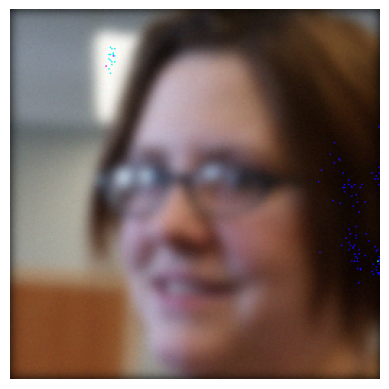} & \includegraphics[width=1.3cm]{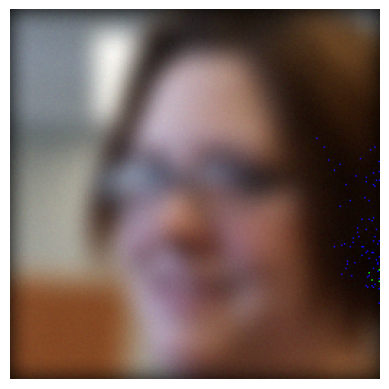}\\
Supervised & \includegraphics[width=1.3cm]{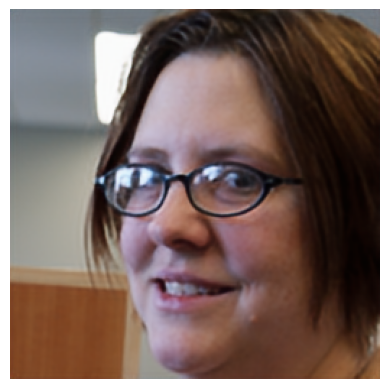} & \includegraphics[width=1.3cm]{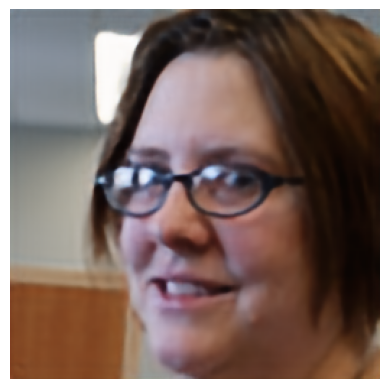} & \includegraphics[width=1.3cm]{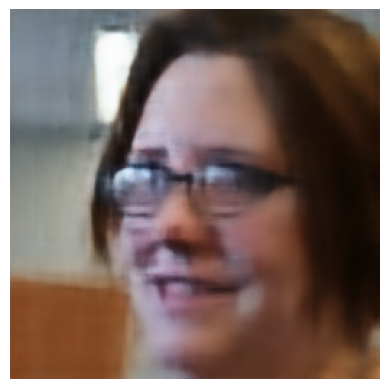} & \includegraphics[width=1.3cm]{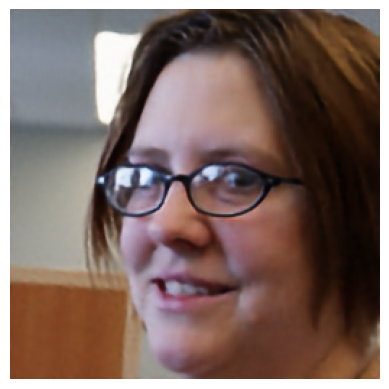} & \includegraphics[width=1.3cm]{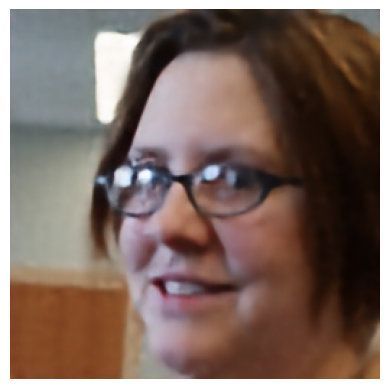} & \includegraphics[width=1.3cm]{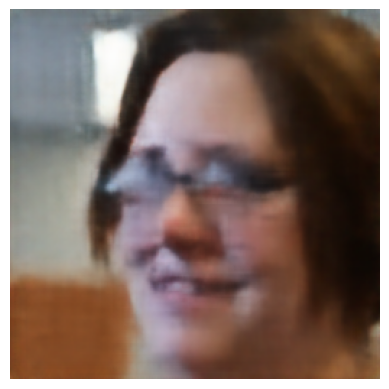}\\
ECALL & \includegraphics[width=1.3cm]{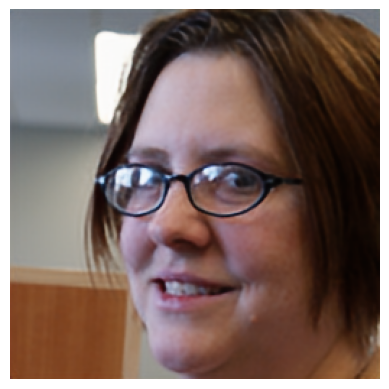} & \includegraphics[width=1.3cm]{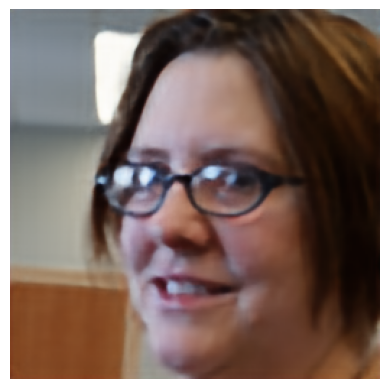} & \includegraphics[width=1.3cm]{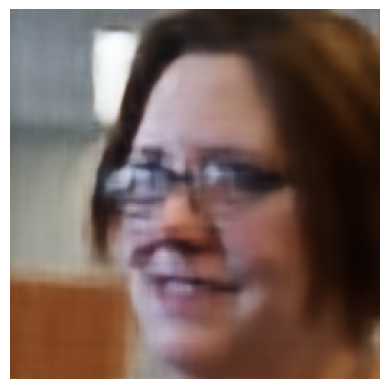} & \includegraphics[width=1.3cm]{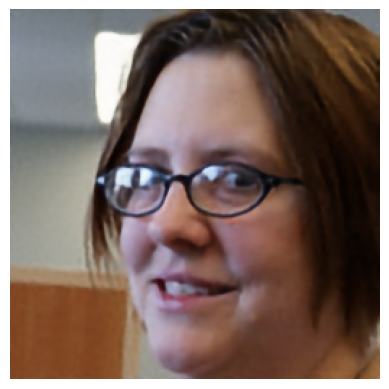} & \includegraphics[width=1.3cm]{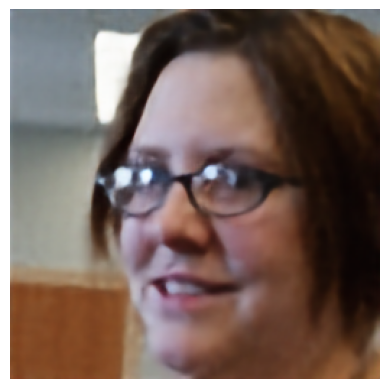} & \includegraphics[width=1.3cm]{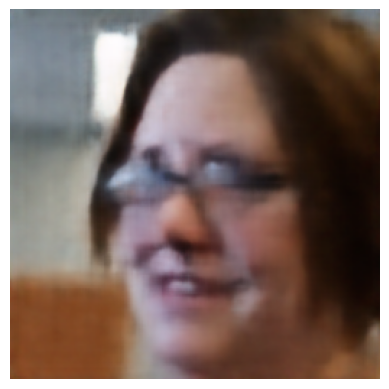} \\
\bottomrule
\end{tabular}}
\caption{Blind deconvolution results using the supervised and proposed unsupervised learning.}
\label{tbl:test}
\end{table}

\newpage
For quantitative evaluation of the kernel estimation, we use the relative $L^2$-norm error $\err \coloneqq \norm{\kernel_w - \kernel^\star}_2 / \norm{\kernel^\star}_2 $ and the maximum of normalized convolution \cite{hu2012good} $\mcn \coloneqq \max [ \kernel_w \ast \kernel^\star / ( \norm{\kernel_w}_2 \norm{\kernel^\star}_2 )] $. The quantitative analysis of the kernel estimation is shown in Table \ref{tbl:kernelerror}. To evaluate the quality of the deconvolved images we use structural similarity index measure (SSIM) and peak signal-to-noise ratio (PSNR). Results are shown in Tables \ref{tbl:kernelerror} and \ref{tbl:testerror}. Overall, we can see that ECALL performs comparably to supervised learning.

\begin{table}[htb!]
\centering
\makebox[0.9\linewidth]{
\begin{tabular}{|c|ccc|ccc|}
\toprule
& &\centering Noiseless & & &\centering Noisy & \tabularnewline
\midrule
 &\centering Broad &\centering Medium &\centering Narrow &\centering Broad &\centering Medium &\centering Narrow \tabularnewline
\midrule
Measurements&\centering $\err$ $\downarrow$ &\centering $\err$ $\downarrow$ &\centering $\err$ $\downarrow$ &\centering $\err$ $\downarrow$ &\centering $\err$ $\downarrow$ &\centering $\err$ $\downarrow$ \tabularnewline
&\centering MNC$\uparrow$ &\centering MNC$\uparrow$ &\centering MNC$\uparrow$ &\centering MNC$\uparrow$ &\centering MNC$\uparrow$ &\centering MNC$\uparrow$\tabularnewline
\midrule
Supervised& $0.0765$ & $0.0129$& $0.0029$& $0.0765$& $0.0124$& $0.0050$\\
& $0.9584$ & $0.9950$ & $0.9937$ & $0.9589$ & $0.9926$ & $0.9892$\\
\midrule
ECALL & $0.0434$ & $0.0224$& $0.0119$ & $0.0571$ & $0.0598$ & $0.0879$ \\
& $0.9991$ & $0.9997$ & $0.9999$ & $0.9984$ & $0.9982$ & $0.9962$\\
\bottomrule
\end{tabular}}
\caption{Quantitative evaluation of kernel estimation using $\err$ (relative $L^2$-error) and $\mcn$ (maximum of normalized convolution). The arrows $\uparrow$ and $\downarrow$ indicate whether higher or lower values indicate better performance. }
\label{tbl:kernelerror}
\end{table}

\begin{table}[htb!]
\centering
\makebox[\linewidth]{
\begin{tabular}{|c|ccc|ccc|}
\toprule
& &\centering Noiseless & & &\centering Noisy & \tabularnewline
\midrule
 &\centering Broad &\centering Medium &\centering Narrow &\centering Broad &\centering Medium &\centering Narrow \tabularnewline
\midrule
Measurements&\centering SSIM$\uparrow$ &\centering SSIM$\uparrow$ &\centering SSIM$\uparrow$ &\centering SSIM$\uparrow$ &\centering SSIM$\uparrow$ &\centering SSIM$\uparrow$ \tabularnewline
&\centering PSNR$\uparrow$ &\centering PSNR$\uparrow$ &\centering PSNR$\uparrow$ &\centering PSNR$\uparrow$ &\centering PSNR$\uparrow$ &\centering PSNR$\uparrow$\tabularnewline
\midrule
Supervised& $0.9437$ & $0.8636$& $0.7080$& $0.9065$& $0.8129$& $0.6684$\\
& $33.43$dB & $29.42$dB & $25.11$dB & $31.70$dB & $27.98$dB & $24.31$dB\\
\midrule
ECALL & $0.9435$ & $0.8549$& $0.6956$ & $0.9061$ & $0.8133$ & $0.6720$ \\
& $33.49$dB & $29.17$dB & $24.88$dB & $31.67$dB & $28.03$dB & $24.31$dB\\
\bottomrule
\end{tabular}}
\caption{Quantitative evaluation of reconstruction results using SSIM and PSNR. The arrows $\uparrow$ and $\downarrow$ indicate whether higher or lower values indicate better performance.}
\label{tbl:testerror}
\end{table}

\section{Conclusion}
\label{sec:conclusion}

Blind deconvolution consists in recovering an original image $\signal$ from its noisy blurred observation $\data = \kernel^\star \ast \signal + \noise$. Blind deconvolution is a highly ill-posed problem, since both the original image $\signal$ and the kernel $\kernel^\star$ need to be estimated. To address the ill-posed nature of the problem, blind deconvolution is typically approached with additional prior information on $\signal$ and $\kernel^\star$. Recently, deep learning-based approaches have shown great promise in blind deconvolution, but most of these approaches require paired datasets of original and blurred images, which are often difficult to obtain. To address this, in this work we propose an unsupervised learning approach (ECALL) to blind deconvolution that uses seperate unpaired collections of original and blurred images. Experimental results demonstrate the feasibility and performance of the proposed algorithm.

In short, the main novelty of ECALL is to construct special statistics using the distributions of $\signal$ and $\data$, respectively, which allows us to determine the kernel from which we can then create virtual supervised data pairs. In particular, we use the expectation of the Fourier transform and its absolute value for kernel estimation. There are a number of interesting follow-ups to our work, including generalisation to other inverse problems with unknown forward operator, development of better statistics, analysis of the influence of empirical data and architecture on the estimates, and comparison with GAN-based methods for blind image deconvolution. 

\section{Acknowledgments}
G. Hwang was supported by the National Research Foundation of Korea (NRF) grant funded by the Korean government (MSIT) (NRF-2021R1F1A1048120).

\bibliographystyle{plain}
\bibliography{references}{}

\end{document}